\documentclass[11pt]{amsart}
\pdfoutput=1
\usepackage[margin=1.25in]{geometry}
\usepackage{amsmath, amssymb}
\usepackage{amsfonts}
\usepackage{mathrsfs}
\usepackage[arrow,matrix,curve,cmtip,ps]{xy}
\usepackage{graphicx}
\usepackage{float}
\usepackage{color}
\usepackage{amsthm}
\usepackage{comment}

\allowdisplaybreaks

\newtheorem{theorem}{Theorem}[section]
\newtheorem{lemma}[theorem]{Lemma}
\newtheorem{proposition}[theorem]{Proposition}
\newtheorem{corollary}[theorem]{Corollary}

\newtheorem*{theorem*}{Theorem}
\theoremstyle{remark}

\numberwithin{equation}{section}

\newcommand{\C}{\mathbb{C}}
\newcommand{\T}{\mathcal{T}}
\renewcommand{\C}{\mathcal{C}}
\newcommand{\Mod}{\mathrm{Mod}}
\newcommand{\Out}{\mathrm{Out}}

\newcommand{\diam}{\mathrm{diam}}

\begin{document}
\title[Ratio Optimizers]{ Pseudo-Anosovs optimizing the ratio of Teichm\"uller to curve graph translation length}

\author{Tarik Aougab and Samuel J. Taylor}

\address{Department of Mathematics \\ Brown University \\ 151 Thayer Street, Providence, RI 02912\\ USA}
\email{tarik\_aougab@brown.edu}

\address{Department of Mathematics \\ Yale University \\ 10 Hillhouse Avenue, New Haven, CT 06510 \\ USA}
\email{s.taylor@yale.edu}

\date{\today}


\begin{abstract}

 Given $\phi$ a pseudo-Anosov map, let $\ell_\T(\phi)$ denote the translation length of $\phi$ in the Teichm\"uller space, and let $\ell_\C(\phi)$ denote the stable translation length of $\phi$ in the curve graph. Gadre--Hironaka--Kent--Leininger showed that, as a function of Euler characteristic $\chi(S)$, the minimal possible ratio $\tau(\phi) = \frac{\ell_\T(\phi)}{\ell_\C(\phi)}$ is $\log(|\chi(S)|)$, up to uniform additive and multiplicative constants. In this short note, we introduce a new construction of such \emph{ratio optimizers} and demonstrate their abundance in the mapping class group.
Further, we show that ratio optimizers can be found arbitrarily deep into the Johnson filtration as well as in the point pushing subgroup. 
\end{abstract}
\maketitle

\section{Introduction}

Let $S=S_{g,p}$ denote the orientable surface of genus $g$ with $p$ punctures, and let $\omega(g,p)=\omega(S)=3g+p-4$ be its \emph{complexity}. Let $\mbox{Mod}(S)$ denote the mapping class group of $S$, $\mbox{Teich}(S)$ the Teichm\"uller space equipped with the Teichm\"uller metric, and $\mathcal{C}(S)$ the curve graph of $S$.

Consider the coarsely-defined map $\pi_{g,p} : \mbox{Teich}(S)\rightarrow \mathcal{C}^0(S)$, which sends a marked hyperbolic surface to the simple closed curve(s) of shortest length. The map $\pi_{g,p}$ was originally studied by Masur-Minsky, who, as part of the proof of the $\delta$-hyperbolicity of $\mathcal{C}(S)$, demonstrated the existence of a constant $K=K(g,p)$ such that $\pi_{g,p}$ is coarsely $K$-- Lipschitz \cite{MasurMinsky}. Recall that a map $f:X \to Y$ between metric spaces is \emph{coarsely $K$--Lipschitz} if there is an $L\ge 0$ such that $d_Y(f(a),f(b)) \le K \cdot d_X(a,b) +L$ for all $a,b \in X$.

Let $K(g,p)$ denote the optimal possible value of the Lipschitz constant for $\pi_{g,p}$ as a function of $S_{g,p}$; that is,
\[ K(g,p)= \inf \{c \in \mathbb{R}: \pi_{g,p} \hspace{1 mm} \mbox{is coarsely $c$-Lipschitz} \}. \]
Gadre--Hironaka--Leininger--Kent showed that $K(S_{g,0})\sim \frac{1}{\log(g)}$ \cite{GHKL}. Thus, not only is $\pi_{g,p}$ coarsely Lipschitz, but it is coarsely contracting, and the optimal contraction factor approaches $0$ as $g \rightarrow \infty$. Following this work, Valdivia showed that for $r \in \mathbb{Q}$ a fixed rational number, the optimal Lipschitz constant for a sequence of surfaces $S_{g_{i}, p_{i}}$ with $g_{i}/p_{i}= r$, also decays logarithmically in complexity (relative to constants that a priori depend on $r$) \cite{valdivia2014lipschitz}. 

In one direction, \cite{GHKL} follow Masur and Minsky's original proof, while controlling the portions of the argument that a priori grow with the complexity of $S$. Conversely, to show that a Lipschitz constant on the order of $1/\log(g)$ is optimal, they construct, for each $g$, a pseudo-Anosov map $\psi \in \mbox{Mod}(S_{g})$ such that the ratio $\tau(\psi)=\ell_\T(\psi)/\ell_\C(\psi)$ of its translation length in $\mbox{Teich}(S)$, denoted $\ell_\T(\psi)$, to its stable translation length in $\mathcal{C}(S)$, denoted $\ell_\C(\psi)$, is on the order of $\log(g)$.  

The purpose of this short note is to give a new construction of pseudo-Anosov maps for which $\tau(\psi)$ is optimal, i.e. on the order of $\log(\omega(S))$. We call such pseudo-Anosovs \textit{ratio optimizers} and our construction shows their abundance in the mapping class group:

\begin{theorem} \label{Infinite} There exists a function $f(\omega)= O(\log(\omega))$, and a Teichm\"uller disk $\mathcal{D} \subset \mbox{Teich}(S_{g,p})$ such that there are infinitely many conjugacy classes of primitive pseudo-Anosovs $\psi$ with $\tau(\psi)=\frac{\ell_\T(\psi)}{\ell_\C(\psi)}< f(\omega(g,p)))$, and the invariant axis of $\psi$ is contained in $\mathcal{D}$. 
\end{theorem}

\noindent We will see in Corollary \ref{cor:opt} that the function $f(\omega)$ can be taken to be $\log(2B \cdot \omega)$ where $B\ge 1$ is a constant not depending on $\omega$.

In addition to establishing the abundance of ratio optimizers, our methods show that ratio optimizers can be constructed in subgroups of mapping class groups which are well-known not to contain pseudo-Anosov mapping classes that minimize Teichm\"uller space translation length alone. In particular, we build ratio optimizers arbitrarily deep into the Johnson filtration as well as in the point pushing subgroup for a mapping class group of a surface with a single puncture.

For a group $\Gamma$, let $\Gamma^{(k)}$ denote the $k$th term of its lower central series. That is $\Gamma^{(1)} = [\Gamma,\Gamma]$ is the commutator subgroup and $\Gamma^{(k+1)} = [\Gamma^{(k)},\Gamma]$. For any $k\ge0$ there is a surjective homomorphism
\[
\Mod(S) \to \Out(\pi_1(S) / \pi_1(S)^{(k)}),
\]
whose kernel, denoted $J_k$, is the \emph{$k$th term of the Johnson filtration}. These subgroups were introduced by Johnson in \cite{johnson1983survey}. Note that $J_1$ is the Torelli subgroup of $\Mod(S)$ and $J_2$ is the so-called Johnson kernel.

\begin{theorem}\label{th:J_filt}
There exists a uniform constant $C_J\ge0 $ satisfying the following. Let $S = S_{g,p}$, with $g\ge 2$ and $p=0$ or $p=1$, and denote by $J_k(S)$ the $k$th term of the Johnson filtration of $\Mod(S)$. Then there exists $\phi_k \in J_k(S)$ with 
$$
\tau(\phi_k) =  \frac{\ell_\T(\phi)}{\ell_\C(\phi)} \le C_J \log \omega(S).
$$
That is, there are ratio optimizers arbitrarily deep into the Johnson filtration.
\end{theorem}

We remark that Theorem \ref{th:J_filt} is entirely different from the situation of minimizing $\ell_{\T}(\phi)$ alone. In fact, Farb--Leininger--Margalit \cite{farb2008lower} have shown that the minimal Teichm\"uller space translation length among pseudo-Anosov mapping class in $J_1$ is uniformly bounded from above and below, independent of genus.  This is in contrast to work of Penner who shows that among all pseudo-Anosov homeomorphisms this quantity is on the order of $1/g$ \cite{penner1991bounds}.

Finally, for a surface $S_{g,1}$ with $g\ge2$, we denote the kernel of the natural map 
\[
\Mod(S_{g,1}) \to \Mod(S_{g,0})
\]
by $PP_g$. This is the point pushing subgroup of $\Mod(S_{g,1})$; it consists of mapping classes which are isotopic to the identity after ignoring the puncture. Similar to the situation discussed above, it is known that pseudo-Anosov mapping classes in  $PP_g$ cannot minimize Teichm\"uller space translation length \cite{dowdall2011dilatation}. However, this is not an issue for ratio optimizers:

\begin{theorem} \label{PP}
There exists a uniform constant $C_P\ge0$ satisfying the following. Let $S= S_{g,1}$ with $g\ge2$ and let $PP_{g} \le \Mod(S)$ be the point pushing subgroup of its mapping class group. Then there is $\phi \in PP_{g}$ with 
\[
\tau(\phi) = \frac{\ell_\T(\phi)}{\ell_\C(\phi)} \le C_P \log \omega(S).
\]
\end{theorem}

\noindent \textbf{Acknowledgments:} Both authors are partially supported by the National Science Foundation. The first by NSF grant DMS- 1502623 and the second by NSF grant DMS-1400498.

\section{Background}
\subsection{Curves, filling pairs and projections.} \label{sec:curves}
Let $S_{g,p}$ denote the genus $g$ surface with $p \geq 0$ punctures. The \textit{complexity} of $S$ is defined as $\omega(S)= \omega(g,p)= 3g+p-4$. For all surfaces in this paper we assume $\omega(S)>0$. A simple closed curve $c$ on $S$ is \textit{essential} if it is not homotopically trivial and if it is not homotopic into a neighborhood of a puncture. Given two essential simple closed curves $\alpha, \beta$ their \textit{geometric intersection number}, denoted $i(\alpha, \beta)$, is defined as 
\[ i(\alpha, \beta):= \min_{x \sim \alpha} |x \cap \beta|, \]
where $\sim$ denotes homotopy. If $|\alpha \cap \beta| = i(\alpha, \beta)$, we say $\alpha$ and $\beta$ are in \textit{minimal position}. Note that any collection of pairwise non-homotopic essential curves can be placed in pairwise minimal position on $S$. Indeed, when $S$ is equipped with any complete hyperbolic metric, any pair of closed geodesics is in minimal position and there exists a unique geodesic in each free homotopy class of essential curve. 

A pair of essential simple closed curves $\alpha, \beta$ are in minimal position on a closed surface $S_{g}$ if and only if no complementary component of $\alpha \cup \beta$ is a \textit{bigon}, a disk whose boundary is comprised of one arc of $\alpha$ and one of $\beta$ \cite{farb2011primer}. 

A collection of curves $\Gamma=\left\{\gamma_{1},..., \gamma_{n}\right\}$ in pairwise minimal position is said to \textit{fill} a surface $S$ if the complement of their union consists of a disjoint union of topological disks and once-punctured disks. Equivalently, $\Gamma$ fills $S$ so long as every essential simple closed curve $\alpha$ has positive geometric intersection number with at least one curve in $\Gamma$. Let $i_{g,p}$ denote the minimum possible geometric intersection number for a filling pair $\alpha, \beta$ on $S_{g,p}$. A simple Euler characteristic argument shows that $i_{g,p}$ must grow linearly in $\omega(g,p)$. In \cite{aougab2015minimally} and \cite{AT1} the quantities $i_{g,p}$ were determined.

\begin{lemma}[Minimally intersecting filling pairs] \label{FillPair} Minimally intersecting filling pairs intersect as follows:

\begin{enumerate}
\item If $g \neq 2,0$ and $p=0$, $i_{g,p}=2g-1$. 
\item If $g \neq 2,0$ and $p \geq 1$, $i_{g,p}= 2g+p-2$. 
\item If $g=0$ and $p \geq 6$ is even, then $i_{g,p}= p-2$. On the other hand if $p$ is odd, $i_{g,p}= p-1$. 
\item If $g=2$ and $p\leq 2$, $i_{g,p}= 4.$
\item If $g=2$ and $p \geq 2$ is even, then $i_{g,p}=2g+p-2$; and if $p \geq 3$ is odd, $i_{g,p} \leq 2g+p-1$. 
\end{enumerate}
\end{lemma}

In our application to pseudo-Anosov mapping classes in the Johnson filtration, we also require information about filling pairs of separating curves. Let $i^{sep}_{g,p}$ denote the minimum geometric intersection number taken over all filling pairs $\alpha, \beta$ where both $\alpha, \beta$ are separating curves. Then we have the following:  

\begin{lemma}[Separating filling pairs] \label{SepPair} There exists a constant $C \geq 0$ such that if  $g \geq 2$ and $p=0$ or $1$, there is a filling pair $(\alpha, \beta)$ on $S_{g,p}$ with both $\alpha, \beta$ separating curves, satisfying $i(\alpha, \beta) \leq C \cdot \omega(g,p)$.
\end{lemma}

\begin{proof} First suppose $p=0$,  and define $\alpha_{2}, \beta_{2}$ to be any pair of separating curves which fill $S_{2}$; similarly let $\alpha_{3}, \beta_{3}$ be any pair of separating curves filling $S_{3}$. These will be the seeds of an inductive construction. 

Now let $\rho, \gamma$ be a pair of simple separating arcs on $S_{2,1}$ (which we interpret as the genus $2$ surface with one boundary component, as opposed to one puncture) having the property that any essential arc in $S_{2,1}$ intersects either $\rho$ or $\gamma$.  
Then given $\alpha_{g}, \beta_{g}$, we form $\alpha_{g+2}, \beta_{g+2}$ as follows: excise a small open disk centered at one of the points in $\alpha_{g} \cap \beta_{g}$. After excising, $\alpha_{g}$ and $\beta_{g}$ have become arcs which we denote by $\tilde{\alpha}_{g}, \tilde{\beta}_{g}$. We then glue on a copy of $S_{2,1}$, matching the endpoints of $\tilde{\alpha}_{g}$ to those of $\gamma$, and similarly matching the endpoints of $\tilde{\beta}_{g}$ to those of $\rho$. We obtain a pair of simple closed curves $\alpha_{g+2}, \beta_{g+2}$ on $S_{g+2}$, and we claim that these curves are both separating and that they fill. 

Note first that $\alpha_{g+2}, \beta_{g+2}$ are in minimal position since no complementary region is a bigon, and therefore it suffices to prove that if $\kappa$ is any essential simple closed curve on $S_{g+2}$, $\kappa$ is not disjoint from $\alpha_{g+2} \cup \beta_{g+2}$. If $\kappa$ can be isotoped into the original copy of $S_{g}$, it must intersect either $\tilde{\alpha}_{g}$ or $\tilde{\beta}_{g}$ since $\alpha_{g}, \beta_{g}$ fill on $S_{g}$. Therefore, we can assume that $\kappa$ projects non-trivially to the copy of $S_{2,1}$; that is, that $\kappa$ intersects this copy of $S_{2,1}$ in at least one arc which is not boundary parallel. This arc must intersect either $\rho$ or $\gamma$ since any arc does so by construction. 
Therefore $\kappa$ intersects $\alpha_{g+2} \cup \beta_{g+2}$ and we conclude that the new pair fills $S_{g+2}$. 

That $\alpha_{g+2}, \beta_{g+2}$ are both separating is immediate since both are obtained by concatenating a pair of separating arcs in disjoint subsurfaces. Finally, $i(\alpha_{g+2}, \beta_{g+2}) \leq i(\alpha_{g}, \beta_{g})+ i(\gamma, \rho)$.

If $p =1$, then by puncturing one complementary region of $S_{g,0} \setminus (\alpha_{g} \cup \beta_{g})$,  $(\alpha_{g}, \beta_{g})$ is a filling pair on $S_{g,1}$ with the desired properties. 
\end{proof}

\subsection{Annular projections and the bounded geodesic image theorem}
For an annulus $Y \subset S$ whose core curve $\alpha$ is essential, let $\tilde{Y}$ be the cover of $S$ associated to the conjugacy class of the cyclic subgroup of $\pi_1(S)$ represented by $\alpha$. Let $\overline{Y}$ be the compactification of $\tilde{Y}$ obtained by choosing a hyperbolic metric on $S$ and lifting it to $\tilde Y$. The curve graph $\C(Y)$ of the annulus $Y$ is the graph whose vertices are homotopy classes of properly embedded, simple arcs of $\overline{Y}$ whose endpoints lie on distinct boundary components. Two vertices $x$ and $y$ of $\C(Y)$ are joined by an edge of $\C(Y)$ if and only if $x$ and $y$ can be represented by arcs in $\overline Y$ with disjoint interiors. There is a projection $\pi_Y$ from the vertices of the curve graph of $S$ to arcs of $\C(Y)$, known as \emph{subsurface projection}. Given $\beta \in \C^0(S)$ realize $\alpha$ and $\beta$ with minimal intersection in $S$. If $\beta$ is disjoint from $\alpha$ then define $\pi_Y(\beta) = \emptyset$. Otherwise, the preimage of $\beta$ in the cover $\tilde{Y}$ contains simple, properly embedded arcs with well-defined endpoints on distinct components of $\partial \overline{Y}$ and we define $\pi_Y (\beta) \subset \C^0(Y)$ to be this collection of arcs in $\overline{Y}$.

If $\alpha \in \C^0(S)$ is a curve which is the core of an annulus $Y$ then we also use the notation $\C(\alpha)$ for the curve complex $\C(Y)$ and we denote its path metric by $d_{\alpha}$. Let $\pi_{\alpha}: \C(S) \setminus N_1(\alpha) \to \C(\alpha)$ be the corresponding subsurface projection, where $N_1(\cdot)$ is the closed $1$-neighborhood in $\C_1(S)$. Also, write $d_\alpha(\beta,\gamma)$ for $\diam_{\C(\alpha)} (\pi_\alpha(\beta) \bigcup \pi_\alpha(\gamma))$.

From \cite{MasurMinsky}, we recall the following:

\begin{lemma}[Masur--Minsky] \label{lem:facts}
Let $S$ be a surface with $\omega(S) > 1$. For $\alpha \in \C^0(S)$ and any path $\gamma=  \gamma_0, \gamma_1, \ldots, \gamma_n$ of curves in $\C(S)$ each intersecting $\alpha$ essentially, we have:
\begin{enumerate}
\item $\diam_{\C(\alpha)} \pi_\alpha(\gamma) \le 1$
\item $d_\alpha(\gamma_0,\gamma_n) \le n+1$ 
\item If $T_\alpha$ is the Dehn twist about $\alpha$, then $d_\alpha(\gamma, T^N(\gamma)) \ge N-2$.
\end{enumerate}
\end{lemma}

Finally, we recall the bounded geodesic image theorem of Masur--Minsky \cite{MM2}. The version we state here is due to Webb and gives a uniform, computable constant \cite{webb2013uniform}. It is stated below for arbitrary subsurfaces $Y \subset S$, but we will use it only for annuli.

\begin{theorem}[Bounded geodesic image theorem] \label{BGIT}
There exists $M \ge 0$ so that for any surface $S$ and any geodesic $g$ in $\C(S)$, if each vertex of $g$ has nontrivial projection to the subsurface $Y$ then $\mathrm{diam}(\pi_Y(g)) \le M$.
\end{theorem}

\subsection{Mangahas' Lemma}

Let $C_A = N_1(\alpha)$ and $C_B = N_1(\beta)$ be $1$-neighborhoods of the curves $\alpha$ and $\beta$ in $\C(S)$ and let $M$ be as in the bounded geodesic image theorem (Theorem \ref{BGIT}). The following is a special case of a combination of Lemma $5.3$ of \cite{mangahas2013recipe} along with the claim used in its proof. Recall that for a word $w$ in the free group $F(a,b)$, the \emph{syllable length} of $w$, denote $|w|_s$, is the number of powers of $a$ or $b$ that occur in the reduced form for $w$.

\begin{lemma}[Mangahas] \label{lem:mang}
Let $a,b$ be powers of Dehn twists about curves $\alpha,\beta$, respectively, such that $d_S(\alpha,\beta)\ge 3$, i.e. $\alpha$ and $\beta$ fill $S$. Suppose that for all $k\ne 0$
\[
d_\alpha(C_B,a^k \cdot C_B)>2M+4 \quad \mathrm{and} \quad d_\beta(C_A, b^k \cdot C_A) > 2M+4.
\]
Then for any word $w$ in $\langle a, b \rangle$, either
\[
d_S(w \cdot \alpha,\alpha) \ge |w|_s \quad \mathrm{or} \quad d_S(w \cdot \beta,\beta) \ge |w|_s.
\]
\end{lemma}

\noindent We remark that if $a= T_\alpha^{l_1}$ and $b=T_\beta^{l_2}$ then by Lemma \ref{lem:facts} the hypotheses of Lemma \ref{lem:mang} are satisfied so long as $|l_1|,|l_2| \ge 2M+7 $.

\subsection{Pseudo-Anosovs and Teichm{\"u}ller disks} \label{sec:teich}
For curves $\alpha$ and $\beta$ which jointly fill the surface $S$ and have intersection number $i(\alpha,\beta) =n$, there is a representation $\Psi: \langle T_\alpha, T_\beta \rangle \to PSL_2(\mathbb{R})$ given by
\begin{align} \label{rep}
\Psi(T_\alpha) = \begin{pmatrix} 1 & n  \\ 0 & 1 \end{pmatrix}, \quad  \Psi(T_\beta) = \begin{pmatrix} 1 & 0  \\ -n & 1 \end{pmatrix}.
\end{align}

Thurston showed that the  pseudo-Anosov mapping classes of the subgroup $ \langle T_\alpha, T_\beta \rangle \le \Mod(S)$ are exactly the ones mapping to hyperbolic matrices in $PSL_2(\mathbb R)$ (i.e. matrices with $2$ distinct eigenvalues). Further, he showed that the dilatation of such a pseudo-Anosov is equal to the largest eigenvalue of its representative matrix \cite{thurston1988geometry}. Since the Teichm\"uller space translation length of a pseudo-Anosov mapping class $\phi$ is equal to the logarithm of its dilatation, this allows a direct computation of $\ell_\T(\phi)$ for $\phi \in \langle T_\alpha, T_\beta \rangle$. For proofs of these facts see \cite{farb2011primer}.

The representations $\Psi: \langle T_\alpha, T_\beta \rangle \to PSL_2(\mathbb{R})$ in (\ref{rep}) comes from the singular flat structure $S(\alpha,\beta)$ on $S$ associated to the pair $(\alpha, \beta)$. This is the structure induced by the quadratic differential $q$ on $S$ whose vertical foliation is equal to $\alpha$ and whose horizontal foliation is equal to $\beta$ as measured foliations.  
Alternatively, one can consider the dual square complex to the graph $\alpha \cup \beta$, which induces a complex structure on $S$ along with the quadratic differential $q$ obtained by taking $dz^2$ in the interior of each square. The Dehn twists $T_\alpha$ and $T_\beta$ can each be realized by an affine map with respect to the singular flat structure, and the ``derivative'' map induces the representation to $PSL_2 (\mathbb{R})$. See \cite{farb2011primer} for details.

The image in $\T(S)$ of the  $SL_2(\mathbb R)$ orbit of a quadric differential $q$ is known as a \emph{Teichm\"uller disk}. Given a filling pair $\alpha$, $\beta$, by $\mathbb{D}(\alpha, \beta)$ we mean the Teichm\"uller disk corresponding to the quadratic differential described above, determined by the dual square complex to $\alpha \cup \beta$. The subgroup of $\Mod(S)$ preserving $\mathbb{D}(\alpha,\beta)$ is known as the Veech group $V(\alpha,\beta)$ and equals the image in $\Mod(S)$ of the affine homeomorphisms of the singular flat surface $S(\alpha,\beta)$. In particular, we note that $\langle T_\alpha, T_\beta \rangle \le V(\alpha,\beta)$.

\section{Ratio optimizers via QI trees}
Let $S = S_{g,p}$. Choose simple closed curves $\alpha ,\beta \in \C^0(S)$ which fill $S$, that is, for which $d_S(\alpha,\beta) \ge 3$. For notational convenience, set $i =i_{\alpha,\beta} = i(\alpha,\beta)$. We will use later that $\alpha$ and $\beta$ can be chosen so that $i(\alpha,\beta) \le i_{g,p}$ where $i_{g,p}$ is as in Section \ref{sec:curves} and depends linearly on the complexity $\omega(S) = \omega(g,p)$ of $S$.
Let $M$ be the bounded geodesic image constant of Theorem \ref{BGIT}. Recall that $M$ is independent of the complexity of $S$. Let $B = 2M+7$ and set $a = T^B_\alpha$ and $b=T^B_\beta$.

\begin{figure}[h]
\begin{center}
\includegraphics[scale= .5]{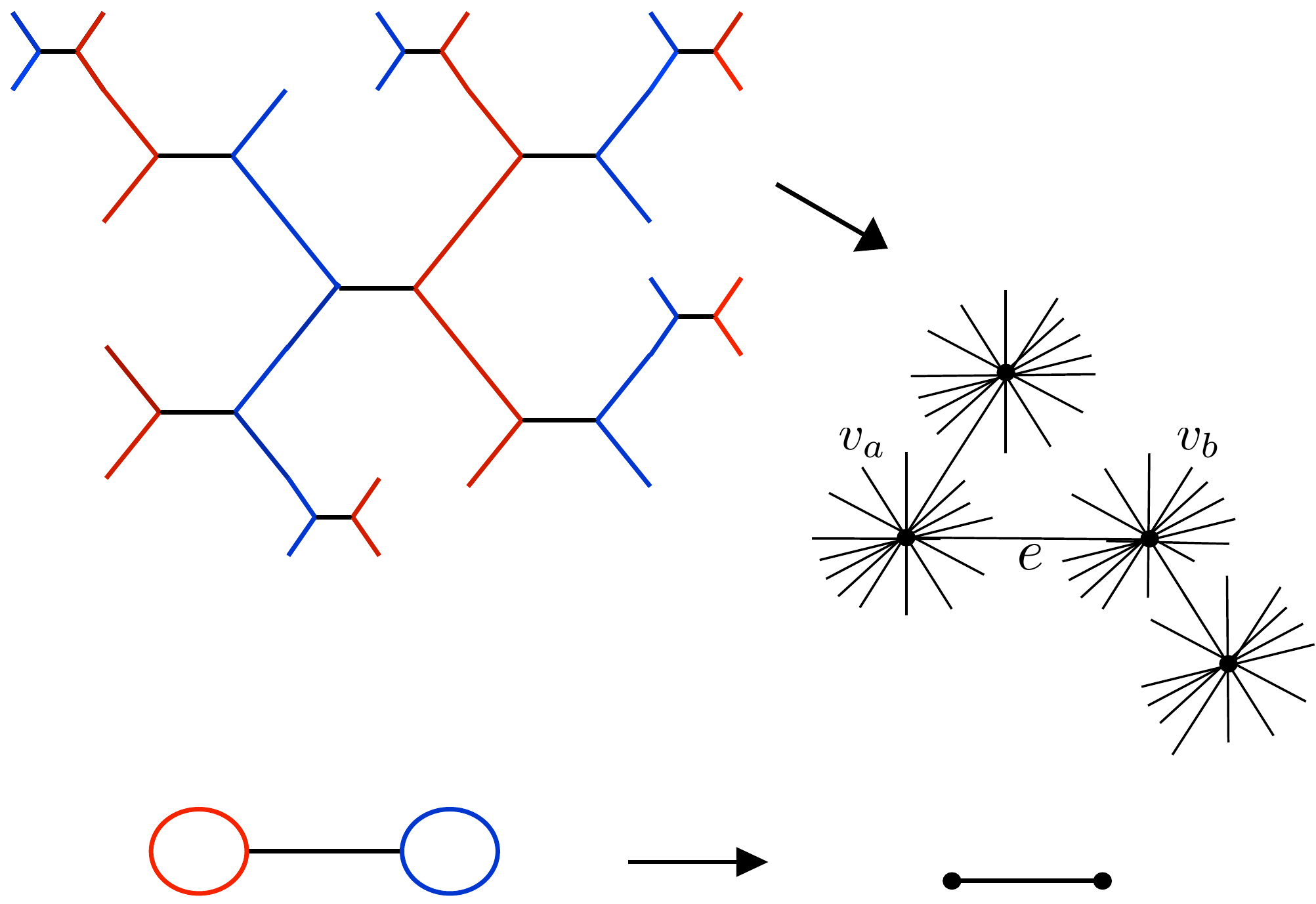}
\caption{The tree $T_{\alpha,\beta}$ as a collapse of the trivalent tree.}
\label{fig:tree}
\end{center}
\end{figure}

It is well known that the subgroup of $\Mod(S)$ generated by $a,b$ is isomorphic to the free group of rank $2$ (see, for example, \cite[Theorem 6.1]{Clein} \cite[Proposition 3.3]{mangahas2010uniform}) and we make the identification $F_2 = \langle a,b \rangle \le \Mod(S)$. Let $T = T_{\alpha,\beta}$ be the Bass--Serre tree for the splitting $\langle a \rangle * \langle b \rangle$. In details, $T$ is the $F_2$--tree obtained by taking the universal cover of the ``barbell'' graph whose loops are labeled by $\{a,b\}$ and collapsing the lift of each $a$--edge and each $b$--edge. See Figure \ref{fig:tree}. Denote the image of the axis for $a$ by $v_a$ and note that this is the unique vertex of $T$ which is fixed by $a$. Similarly, denote the image of the axis for $b$ by $v_b$. Note that these vertices are joined by an edge, which we call $e$, and $e$ is a fundamental domain for the action $F_2 \curvearrowright T$.

We now define an equivariant map $\mathcal{O}:T \to \C(S)$. The vertex $v_a$ is mapped to $\alpha$,  $v_b$ is mapped to $\beta$ and $e$ is mapped to a fixed geodesic from $\alpha$ to $\beta$ in $\C(S)$. Using the identification $F_2 \to \langle a, b \rangle \le \Mod(S)$, extend the map to all of $T$ by equivariance. This is well-defined since $a$ fixes $\alpha$ and $b$ fixes $\beta$. The main result of this section comes from the following proposition:

\begin{proposition}[The tree quasi-isometrically embeds] \label{prop:tree_qi}
With notation as above, the $\langle a,b \rangle$--equivariant map $\mathcal{O}: T_{\alpha,\beta} \to \C(S)$ is a $(3,7)$--quasi-isometric embedding.
\end{proposition}

We remark that by examining the proof of Lemma \ref{lem:mang}, the constants in Proposition \ref{prop:tree_qi} can be improved upon. This, however, will not be necessary for our application.

Before turning to the proof of Proposition \ref{prop:tree_qi}, we first make a few remarks on distance in the (infinite valence) tree $T$. For a reduced word $w$ in $\langle a,b \rangle$, a \emph{syllable} of $w$ is a maximal subword of the form $a^k$ or $b^k$. The \emph{syllable length} of $w$, denoted $|w|_{s}$, is the number of syllables in the word $w$. For example, for the reduced word $w= a^{k_1}b^{k_2} \ldots a^{k_l} $, $|w|_s = l$. The syllable spelling of $w$ is exactly the normal form associated to the tree $T$ and, hence, can be used to compute distance in $T$. In particular, let $x$ and $y$ be two vertices of $T$; there are four cases depending on whether $x$ and $y$ are in the orbit of $v_a$ or $v_b$. For example, suppose that $x = v_a$ and that $y \in F \cdot v_a \setminus \{v_a\}$. Then there is $w \in F$ with $y = w\cdot x$ and we can write $w= a^{l_1}w'a^{l_2}$ as a reduced syllable decomposition. Now it is easily seen that $d_T(x,y) = d_T(v_a,w'v_a)= |w'|_s+1$. If $y\in F \cdot v_b$ then write $y = w \cdot v_b$ with $w = a^{l_1}w'b^{l_2}$ as a reduced syllable decomposition. Again we see that $d_T(x,y) = d_T(v_a,w'v_b) = |w'|_s+1$. These elementary observations will be used in the proof of Proposition \ref{prop:tree_qi}.

\begin{proof}[Proof of Proposition \ref{prop:tree_qi}]
Recall the definition of $C_{A}, C_{B}$ from the statement of Lemma \ref{lem:mang}. Let $x,y \in T$ and set $\gamma = \mathcal{O}(x)$ and $\delta = \mathcal{O}(y)$. Using equivariance and the fact that $d_{S}(\alpha,\beta) = 3$ we easily see that, $d_{S}(\gamma, \delta) \le 3 \cdot d_T(x,y)$. 

For the other inequality, we may assume (by equivariance) that $x$ equals either $v_a$ or $v_b$; since the proofs in each case are identical we assume that $x= v_a$. First, suppose that $y$ is in the orbit of $v_a$, i.e. that $y = w \cdot v_a$ for $w \in F$. By the definition of $\mathcal{O}$ and the triangle inequality, 
\begin{align*}
d_{S}(\gamma,\delta) &= d_{S}(\alpha, w \cdot \alpha)\\
&\ge d_{S}(\beta, w\cdot \beta)-6, 
\end{align*}
and thus by Lemma \ref{lem:mang}, $d_{S}(\gamma, \delta) \ge |w|_s -6 \ge d_{T}(x,y)-7$. 

If on the other hand $y$ is in the obit of $v_b$, then we choose $w \in F$ so that $y = w \cdot v_b$. By the triangle inequality $d_{S}(\alpha, w \cdot \beta) \ge d_{S}(\alpha , w \cdot \alpha) -3$ and $d_{S}(\alpha, w \cdot \beta) \ge d_{S}(\beta , w \cdot \beta) -3$. Hence, we may apply Lemma \ref{lem:mang} to conclude
\begin{align*}
d_{S}(\gamma,\delta) &= d_{S}(\alpha, w \cdot \beta)\\
&\ge |w|_s - 3 \\
&\ge d_T(x,y) -4.
\qedhere
\end{align*}
\end{proof}

We will say that $w \in \langle a,b \rangle$ is \emph{cyclically reduced} if it has the smallest syllable length among any of its conjugates. The following is an immediate corollary of Proposition \ref{prop:tree_qi}.
\begin{corollary}
 \label{cor:curve_trans}
Let $w \in \langle a,b \rangle \le \Mod(S)$ which is cyclically reduced. Then 
\[
|w|_s \le \ell_\C(w) \le 3|w|_s.
\]  
\end{corollary}

The next lemma is elementary and is used to bound the stretch factors of pseudo-Anosovs obtained by iterating composition. 

\begin{lemma}
 \label{lem:trace}
For any $i \ge2$, let 
$$a = \begin{pmatrix} 1 & i  \\ 0 & 1 \end{pmatrix}, \quad  b = \begin{pmatrix} 1 & 0  \\ -i& 1 \end{pmatrix}
$$
and set $w = a^{\epsilon_1}b^{\delta_1} \ldots a^{\epsilon_{k}}b^{\delta_{k}}$, where $\epsilon_i, \delta_i \in \{\pm1 \}$. Then $\mathrm{trace}(w) \le (2i)^{|w|}$.
\end{lemma}

\begin{proof}
For a $2$-by-$2$ matrix $A$, let $| A |_1$ denote its $l^1$-norm and  $| A|_2$ its $l^2$-norm. For any such $A$,  $|A|_1 \le 2 |A|_2$ . Moreover, for any matrices $A$ and $B$, $|AB|_2 \le |A|_2|B|_2$. Hence,
\begin{align*}
\mathrm{trace}(w) &\le  |w|_1 \\
&\le 2|w|_2 \\
&\le 2 \prod_i |a^{\epsilon_i}b^{\delta_{i}} |_2  \\
&\le 2 (i^2+1)^{k} < (2i)^{2k}.
\qedhere
\end{align*}
\end{proof}

\begin{theorem}[Ratio bounds] \label{StretchBound} Let $\alpha, \beta$ be a filling pair of simple closed curves on $S$, and set $a= T^{B}_{\alpha}, b= T^{B}_{\beta}$. Let $w$ be a cyclically reduced word in $a,b$ satisfying $|w|= |w|_{s}$. Then 
\begin{align*}
&\ \tau(w)= \frac{\l_{\T}(w)}{\l_{\C}(w)} \leq \log(2B\cdot i(\alpha, \beta)). 
\end{align*}
\end{theorem}

\begin{proof}
Recall, as noted in Section \ref{sec:teich}, $\ell_\T(w)$ is equal to the logarithm of largest eigenvalue of the matrix corresponding to $w$.
Applying Lemma \ref{lem:trace} and Corollary \ref{cor:curve_trans}, we compute
\begin{align*}
\tau (w) &\le \frac{\log((2i)^{|w|})}{|w|_s}\\
&=\frac{|w|}{|w|_s}  \log(2i).
\end{align*}
Since $|w| =|w|_s$ and $i = B\cdot i(\alpha,\beta)$, this completes the proof.
\end{proof}

The following corollary of Theorem \ref{StretchBound} gives our construction of ratio optimizers.
\begin{corollary}[Ratio optimizers]  \label{cor:opt}
Let $\alpha, \beta$ be a filling pair of simple closed curves satisfying $i(\alpha, \beta) = \eta_{S}$. Then for $w$ a cyclically reduced word as in Theorem \ref{StretchBound},
\begin{align*} 
\tau(w) = \frac{\l_{\T}(w)}{\l_{\C}(w)}  &\leq \log(B \cdot \eta_{S}) \\
&\leq \log(2B \cdot \omega(S)). 
\end{align*}
\end{corollary}

\section{Counting ratio optimizers in a Teichm\"uller Disk}

In this section, we show that our construction yields infinitely many ratio optimizers whose maximal cyclic subgroups in $\mbox{Mod}(S)$ are pairwise non-conjugate. That is, we will exhibit infinitely many ratio optimizers $\phi_{1},\phi_{2},...$ such that for each $i \neq j$, no power of $\phi_{i}$ is conjugate to a power of $\phi_{j}$. Since each of our ratio optimizers is contained in the group generated by $T_{\alpha}, T_{\beta}$ for $\alpha, \beta$ a minimally intersecting filling pair, it will follow that the Teichm{\"u}ller disk $\mathbb{D}(\alpha, \beta)$ is stabilized by infinitely many primitive, pairwise non-conjugate ratio optimizers. This will complete the proof of Theorem \ref{Infinite}.

To begin, let $w_{1}, w_{2},...$ be an infinite collection of distinct cyclically reduced words in $a,b$, satisfying:

\begin{enumerate}
\item $w_{i} \neq a^{\pm}, b^{\pm}$ 

\item $|w_{i}|= |w_i|_{s}$.

\end{enumerate} 

Since each $w_{i}$ is cyclically reduced and all words in the collection are distinct, the words are pairwise non-conjugate. Furthermore, by property $(1)$, each word $w_{i}$ has nonzero translation length on the tree $T_{\alpha, \beta}$, and therefore by Corollary \ref{cor:curve_trans}, each corresponds to a pseudo-Anosov mapping class under the map sending the free group generated by $a,b$ to the subgroup generated by the Dehn twists $T^B_{\alpha}, T^B_{\beta}$. 

We will refer to the pseudo-Anosov image of $w_{i}$ by $\mathcal{O}(w_{i})$. By Proposition \ref{prop:tree_qi}, $\mathcal{O}(w_{j})$ admits a uniformly quasigeodesic axis $A_{j}$. By property $(2)$ above, we may pass to a subsequence such that for each $k\ge0$ the initial subword of $w_i$ of length $k$ is eventually constant as $i\to \infty$. Translating this fact to the tree $T$ and possibly passing to a further subsequence, we have the following property: the axis of $w_{i}$ in $T$ shares a segment centered around the origin of length at least $i$ with the axis for $w_{i-1}$. Thus, there exists a bi-infinite quasi-geodesic $\mathcal{R}$ in $\mathcal{C}(S)$ and a point $x \in \mathcal{R}$ so that $\mathcal{O}(w_{i})$ admits an axis that shares a segment of length $i$ with $\mathcal{R}$, centered about $x$. Furthermore, as a consequence of Corollary \ref{cor:curve_trans} these axes do not fellow travel in $\mathcal{C}(S)$.

Now, let $\Gamma_{1} \subset \left\{\mathcal{O}(w_{2}), \mathcal{O}(w_{3}),...\right\}$ denote the set of words whose maximal cyclic subgroups are conjugate in $\mbox{Mod}(S)$ to the maximal cyclic subgroup determined by $\mathcal{O}(w_{1})$. By hyperbolicity of $\mathcal{C}(S)$, there is a constant $K>0$ so that for $\mathcal{O}(w_{i}) \in \Gamma_{1}$, there is a conjugator $c_{i} \in \mbox{Mod}(S)$ such that the quasigeodesic $c_{i} \cdot A_{i}$ $K$-fellow travels with $A_{1}$. Let $\l$ denote the stable translation length of $\mathcal{O}(w_{1})$. 

It follows that there exists a uniform constant $r$ depending only on $\l$, the hyperbolicity constant for $\mathcal{C}(S)$, and the quasigeodesic constants determined in Proposition \ref{prop:tree_qi}, so that for any two points $t,s$ on $c_{i} \cdot A_{i}$, there exists a power of $\mathcal{O}(w_{1})$ sending $t$ within $r$ of $s$. Hence the same is true for any two points on $A_{i}$, after replacing $\mathcal{O}(w_{1})$ with its conjugate by $c_{i}^{-1}$. 

We first show that $|\Gamma_{1}| <\infty$. Assume by contradiction that $\Gamma_{1}$ is infinite. Then there exists $\mathcal{O}(w_{i}) \in \Gamma_{1}$ with $i$ arbitrarily large. Choose such an $i \gg 1$, and let $y,z$ denote the endpoints of the segment of $\mathcal{R}$ that $A_{i}$ shares. Note that by construction $y,z \in A_j$ for all $j\ge i$.

Then for any $j>i$, $c_{i}^{-1}c_{j}$ sends $A_{j}$ to $A_{i}$, and post-composing this with some power $e_{j}$ of $c_{i}^{-1} \mathcal{O}(w_{1}) c_{i}$ sends each of $y$ and $z$ within $r$ of themselves. By acylindricity of the action of $\mbox{Mod}(S)$ on $\mathcal{C}(S)$ \cite{bowditch2008tight}, there are at most finitely many mapping classes with this property. Hence, \[ \left\{ (c_{i}^{-1} \mathcal{O}(w_{1})^{e_{j}} c_{i}) c_{i}^{-1} c_{j}\right\}_{j>i}:= \left\{c'_{j} \right\}_{j>i}\]
is a finite collection of mapping classes. We note that $c_{i}^{-1} c_{j} \neq c_{i}^{-1}c_{k}$ for any $j \neq k$, as the axes $A_{j}$ and $A_{k}$ do not fellow travel, and the inverse of $c_{i}^{-1}c_{j}$ sends the endpoints at infinity of $A_{i}$ to those of $A_{j}$. Moreover, $c'_{j}$ is obtained from $c_{i}^{-1}c_{j}$ by post-composition with a map that fixes the endpoints of $A_{i}$, and therefore $c'_{j} \neq c'_{k}$ for $k \neq j$. Thus $|\Gamma_{1}|<\infty$.

It follows that we may pass to a subsequence so that no corresponding pseudo-Anosov determines the same maximal cyclic subgroup up to conjugacy as $\mathcal{O}(w_1)$.
Now we simply iterate this argument. By the exact same logic, the set $\Gamma_{i}$ of all maps in our collection which determine the same (up to conjugacy) maximal cyclic subgroup as $\mathcal{O}(w_{i})$ is finite, and thus 
we can pass to a further subsequence all of whose terms determine pseudo-Anosov mapping classes which are distinct, up to conjugacy and powers, from those already obtained.

\section{Ratio optimizers in the Johnson filtration and point pushing subgroups}
Fix $S = S_{g,p}$ with $g\ge2$ and $p\in \{0,1\}$ and let $\alpha$ and $\beta$ be separating curves of $S$ which fill and intersect minimally, i.e. $i(\alpha,\beta) = i^{sep}_{g,p}$. Recall that by Lemma \ref{SepPair}, there is a constant $C \ge0$, independent of $S$, such that $i(\alpha_{g}, \beta_{g}) \leq C \cdot \omega(g,p)$. Let $a = T^B(\alpha)$ and $b=T^B(\beta)$.

\begin{theorem}
There is a constant $C_J \ge0$  satisfying the following.
Let $S = S_{g,0}$ or $S_{g,1}$ with $g\ge 2$ and denote by $J_k(S)$ the $k$th term of the Johnson filtration of $\Mod(S)$. Then there exist $f_k \in J_k(S)$ with 
$$
\tau(f_k) = \frac{\ell_\T(f_k)}{\ell_\C(f_k)} \le C_J \log \omega(S).
$$
In other words, there are ratio optimizers arbitrarily deep into the Johnson filtration.
\end{theorem}

\begin{proof}
Set $w_1 = aba$ and $w_2 = bab$ for $a$ and $b$ as defined at the beginning of this section. Set 
$$
f_k= [ \ldots [[w_1,w_2],w_1] \ldots w_*]
$$
which is $k$ iterated commutators alternating between $w_1$ and $w_2$. Note that by construction $|f_k| = |f_k|_s$, i.e. each syllable of $f_k$ has length $1$. Since by definition $a,b \in J_1$ the same is true for $w_1,w_2$. Further, as the Johnson filtration $\{J_k\}$ is a central series, see \cite{BaLu} and \cite{morita1991structure}, we have $f_k \in J_k$.

By Corollary \ref{cor:curve_trans}, $\ell_\C(f_k) \ge \frac{1}{3} |f_k|_s =  \frac{1}{3} |f_k|$. Moreover, using Lemma \ref{lem:trace} we can directly compute an upper bound for the dilatation, which (up to a uniform constant) is a product of $|f_{k}|$ with $\log(\omega(S))$. This completes the proof.  
\end{proof}

We now construct ratio optimizers in the point-pushing subgroup $PP_g < \mbox{Mod}(S_{g,1})$ of the mapping class group of a once-punctured surface. To achieve this, it suffices to construct a pair $(\alpha, \beta)$ of curves on $S_{g,1}$ which $(1)$ fill the surface, $(2)$ have geometric intersection number at most some fixed polynomial function of $g$, and $(3)$ such 
that $\alpha$ and $\beta$ are isotopic after forgetting the puncture. Assuming the existence of such a pair, note that the pseudo-Anosov $T^{B}_{\alpha} T_{\beta}^{-B}$ lies in $PP_g$, and by Theorem \ref{StretchBound}, it will be a ratio optimizer. From this, we will obtain:

\begin{theorem} \label{PPsection}
There exists a uniform constant $C_P\ge0$ satisfying the following. Let $S= S_{g,1}$ with $g\ge2$ and let $PP_{g} \le \Mod(S)$ be the points pushing subgroup of its mapping class group. Then there is $\phi \in PP_{g}$ with 
\[
\tau(\phi) = \frac{\ell_\T(\phi)}{\ell_\C(\phi)} \le C_P \log \omega(S).
\]
\end{theorem}

To construct the desired filling pair, begin with a filling pair $(\rho, \delta)$ of non-separating curves on a closed surface $S_{g,0}$ with $i(\rho, \delta)$ bounded above by some fixed linear function of $g$. For example, $(\rho, \delta)$ could be a minimally intersecting filling pair on $S_{g,0}$.

Let $\delta_{1}, \delta_{2}$ be two parallel copies of $\delta$, and puncture the surface $S_{g,0}$ on the interior of the annulus bounded by $\delta_{1}$ and $\delta_{2}$ to form the surface $S = S_{g,1}$. Note that $f_{\delta}:=T_{\delta_{1}}^{3} \circ T_{\delta_{2}}^{-3}$ is a point-pushing map in $\mbox{Mod}(S_{g,1})$. We claim that $\rho$ fills with $f_{\delta}(\rho)$, and that $i(\rho, f_{\delta}(\rho))$ is bounded above by a quadratic function of $g$. 

We first show that these two curves jointly fill $S_{g,1}$; that is we must show that if $\gamma$ is any essential simple closed curve, $\gamma$ must intersect either $\rho$ or $f_{\delta}(\rho)$. We use the following inequality as seen in \cite{ivanov92}:

\begin{lemma} \label{TwistInequality}
Let $c_{1},...,c_{m}$ be a collection of pairwise disjoint, pairwise non-homotopic simple closed curves on a surface $S$ with negative Euler characteristic, let $\textbf{S}:=(s_{1},... s_{m}) \in \mathbb{Z}^{m}$, and let $T^{\mathbb{S}}$ denote the composition of Dehn twists $c_{1}^{s_{1}} \circ...\circ c_{m}^{s_{m}}$. Then for any simple closed curves $\gamma, \rho$, 
\begin{align*}
 \sum_{i=1}^{m} (|s_{i}-2|) i(\rho, c_{i}) i(c_{i}, \gamma)-  i(\rho, \gamma) &\leq i(T^{\textbf{S}}(\rho), \gamma)\\
 &\le \sum_{i=1}^{m} |s_{i}| i(\rho, c_{i}) i(c_{i}, \gamma)+  i(\rho, \gamma).
 \end{align*}

\end{lemma}

Now suppose $i(\gamma, \rho)=0$. Then since $\rho$ fills with $\delta$ on the original closed surface, it follows that $i(\delta_{j}, \gamma) \neq 0$ for $j= 1,2$. Thus the left hand side of the inequality of Lemma \ref{TwistInequality} is non-zero, so $\gamma$ must intersect $f_{\delta}(\rho)$. 

The quadratic bound on $i(\rho, f_{\delta}(\rho))$ follows from the linear bound on $i(\rho, \delta)$, and another application of 
Lemma \ref{TwistInequality}. This completes the proof of Theorem \ref{PPsection}.

\bibliography{biblio.bbl}

\providecommand{\bysame}{\leavevmode\hbox to3em{\hrulefill}\thinspace}
\providecommand{\MR}{\relax\ifhmode\unskip\space\fi MR }
\providecommand{\MRhref}[2]{%
  \href{http://www.ams.org/mathscinet-getitem?mr=#1}{#2}
}
\providecommand{\href}[2]{#2}
\begin{thebibliography}{GHKL13}

\bibitem[AH15]{aougab2015minimally}
Tarik Aougab and Shinnyih Huang, \emph{Minimally intersecting filling pairs on
  surfaces}, Algebr. Geom. Topol. \textbf{15} (2015), no.~2, 903--932.

\bibitem[AT14]{AT1}
Tarik Aougab and Samuel~J Taylor, \emph{Small intersection numbers in the curve
  graph}, Bull. London Math. Soc. \textbf{46} (2014), no.~5, 989--1002.

\bibitem[BL94]{BaLu}
Hyman Bass and Alexander Lubotzky, \emph{Linear-central filtrations on groups},
  Contemporary Math. \textbf{169} (1994), 45--45.

\bibitem[Bow08]{bowditch2008tight}
Brian~H Bowditch, \emph{Tight geodesics in the curve complex}, Invent. Math.
  \textbf{171} (2008), no.~2, 281--300.

\bibitem[Dow11]{dowdall2011dilatation}
Spencer Dowdall, \emph{Dilatation versus self-intersection number for
  point-pushing pseudo-anosov homeomorphisms}, J. Topol. \textbf{4} (2011),
  no.~4, 942--984.

\bibitem[FLM08]{farb2008lower}
Benson Farb, Christopher~J Leininger, and Dan Margalit, \emph{The lower central
  series and pseudo-anosov dilatations}, American Journal of Mathematics
  \textbf{130} (2008), no.~3, 799--827.

\bibitem[FM11]{farb2011primer}
Benson Farb and Dan Margalit, \emph{A primer on mapping class groups},
  Princeton University Press, 2011.

\bibitem[GHKL13]{GHKL}
V~Gadre, E~Hironaka, RP~{Kent, IV}, and CJ~Leininger, \emph{Lipschitz constants
  to curve complexes}, Mathematical Research Letters \textbf{20} (2013), no.~4.

\bibitem[Iva92]{ivanov92}
Nikolai~V. Ivanov, \emph{Subgroups of {T}eichm\"uller modular groups},
  Translations of Mathematical Monographs, vol. 115, American Mathematical
  Society, Providence, RI, 1992, Translated from the Russian by E. J. F.
  Primrose and revised by the author. \MR{1195787 (93k:57031)}

\bibitem[Joh83]{johnson1983survey}
Dennis Johnson, \emph{A survey of the torelli group}, Contemporary Math
  \textbf{20} (1983), 165--179.

\bibitem[Lei04]{Clein}
Christopher~J Leininger, \emph{On groups generated by two positive
  multi-twists: Teichm{\"u}ller curves and lehmer's number}, Geom. Topol.
  \textbf{8} (2004), no.~3, 1301--1359.

\bibitem[Man10]{mangahas2010uniform}
Johanna Mangahas, \emph{Uniform uniform exponential growth of subgroups of the
  mapping class group}, Geom. Funct. Anal. \textbf{19} (2010), no.~5,
  1468--1480.

\bibitem[Man13]{mangahas2013recipe}
\bysame, \emph{A recipe for short-word pseudo-anosovs}, American Journal of
  Mathematics \textbf{135} (2013), no.~4, 1087--1116.

\bibitem[MM99]{MasurMinsky}
Howard~A. Masur and Yair~N. Minsky, \emph{Geometry of the complex of curves.
  {I}. {H}yperbolicity}, Invent. Math. \textbf{138} (1999), no.~1, 103--149.

\bibitem[MM00]{MM2}
Howard~A Masur and Yair~N Minsky, \emph{Geometry of the complex of curves {II}:
  {H}ierarchical structure}, Geom. Funct. Anal. \textbf{10} (2000), no.~4,
  902--974.

\bibitem[Mor91]{morita1991structure}
Shigeyuki Morita, \emph{On the structure of the torelli group and the casson
  invariant}, Topology \textbf{30} (1991), no.~4, 603--621.

\bibitem[Pen91]{penner1991bounds}
Robert~C Penner, \emph{Bounds on least dilatations}, Proc. Amer. Math. Soc
  \textbf{113} (1991), no.~2, 443--450.

\bibitem[Thu88]{thurston1988geometry}
William~P Thurston, \emph{On the geometry and dynamics of diffeomorphisms of
  surfaces}, Bull. American Math. Soc. \textbf{19} (1988), no.~2, 417--431.

\bibitem[Val14]{valdivia2014lipschitz}
Aaron~D Valdivia, \emph{Lipschitz constants to curve complexes for punctured
  surfaces}, arXiv preprint arXiv:1409.2804 (2014).

\bibitem[Web13]{webb2013uniform}
Richard~CH Webb, \emph{Uniform bounds for bounded geodesic image theorems}, J.
  Reine Angew. Math. (2013).

\end{thebibliography}
\bibliographystyle{amsalpha}

\end{document}